\documentclass[10pt]{article}

\begin{document}

\newtheorem{theorem}{Theorem}[section]
\newtheorem{corollary}[theorem]{Corollary}
\newtheorem{definition}[theorem]{Definition}
\newtheorem{conjecture}[theorem]{Conjecture}
\newtheorem{question}[theorem]{Question}
\newtheorem{lemma}[theorem]{Lemma}
\newtheorem{proposition}[theorem]{Proposition}
\newtheorem{example}[theorem]{Example}
\newtheorem{problem}[theorem]{Problem}
\newenvironment{proof}{\noindent {\bf
Proof.}}{\rule{3mm}{3mm}\par\medskip}
\newcommand{\remark}{\medskip\par\noindent {\bf Remark.~~}}
\newcommand{\pp}{{\it p.}}
\newcommand{\de}{\em}

\newcommand{\JEC}{{\it Europ. J. Combinatorics},  }
\newcommand{\JCTB}{{\it J. Combin. Theory Ser. B.}, }
\newcommand{\JCT}{{\it J. Combin. Theory}, }
\newcommand{\JGT}{{\it J. Graph Theory}, }
\newcommand{\ComHung}{{\it Combinatorica}, }
\newcommand{\DM}{{\it Discrete Math.}, }
\newcommand{\ARS}{{\it Ars Combin.}, }
\newcommand{\SIAMDM}{{\it SIAM J. Discrete Math.}, }
\newcommand{\SIAMADM}{{\it SIAM J. Algebraic Discrete Methods}, }
\newcommand{\SIAMC}{{\it SIAM J. Comput.}, }
\newcommand{\ConAMS}{{\it Contemp. Math. AMS}, }
\newcommand{\TransAMS}{{\it Trans. Amer. Math. Soc.}, }
\newcommand{\AnDM}{{\it Ann. Discrete Math.}, }
\newcommand{\NBS}{{\it J. Res. Nat. Bur. Standards} {\rm B}, }
\newcommand{\ConNum}{{\it Congr. Numer.}, }
\newcommand{\CJM}{{\it Canad. J. Math.}, }
\newcommand{\JLMS}{{\it J. London Math. Soc.}, }
\newcommand{\PLMS}{{\it Proc. London Math. Soc.}, }
\newcommand{\PAMS}{{\it Proc. Amer. Math. Soc.}, }
\newcommand{\JCMCC}{{\it J. Combin. Math. Combin. Comput.}, }
\newcommand{\GC}{{\it Graphs Combin.}, }
\thispagestyle{empty}
\title{The Terminal Wiener Index of Trees with  Diameter or Maximum Degree\thanks{
 Supported by National Natural Science Foundation of China
(No.11271256), Innovation Program of Shanghai Municipal Education Commission (No.14ZZ016) and Specialized Research Fund for the Doctoral Program of Higher Education (No.20130073110075). \newline \indent
$^{\dagger}$Corresponding  author ({\it E-mail address:}
xiaodong@sjtu.edu.cn)}}

\author{ Ya-Hong Chen$^{1,2}$ \\
{\small $^1$Department of Mathematics, and MOE-LSC,}
{\small Shanghai Jiao Tong University} \\
{\small  800 Dongchuan road, Shanghai, 200240,  P.R. China}\\
 {\small $^2$Department of Mathematics},
{\small Lishui University} \\
{\small  Lishui, Zhejiang 323000, PR China}\\
Xiao-Dong Zhang$^1$$^{\dagger}$\\
{\small $^1$Department of Mathematics, and MOE-LSC,}
{\small Shanghai Jiao Tong University} \\
{\small  800 Dongchuan road, Shanghai, 200240,  P.R. China}
}
\date{}

\maketitle
 \thispagestyle{empty}
 \begin{minipage}{4in}
 \begin{center}
 Abstract
 \end{center}
   The terminal Wiener index of a tree is the sum of distances for all pairs of pendent vertices, which recently arises in the study of phylogenetic tree reconstruction and the neighborhood of trees.  This paper presents a sharp upper and lower bounds for the terminal Wiener index in terms of its order and diameter and characterizes all extremal trees which attain these bounds. In addition, we investigate the properties of extremal trees which attain the maximum terminal Wiener index among all trees of order $n$ with fixed maximum degree.
   \end{minipage}
 \vskip 0.5cm
 {{\bf Key words:} Terminal Wiener index; Tree; Diameter; Maximum degree.
      }
\vskip 0.5cm
\section{Introduction}
 Many topological indices (molecular-structure descriptors) have been put forward in different studies, from biochemistry to pure mathematics. The Wiener index, which is one of the oldest and most widely used indices in quantitative structure-activity relationships, has been received great attention by mathematicians and chemists (for example, see \cite{dobrynin2001,Gutman1997-k,gutman-k-2, wang2008,wiener1947}).
    Recently, some researchers considered terminal distance matrix \cite{horvat2008,randic2007} and found that it was used in the mathematical
modelling of proteins and genetic \cite{horvat2008,randic2007,randic2004} and regarded it as a source of novel  molecular-structure descriptors \cite{randic2007,smolenskii2009}.
Due to study  on the terminal distance matrix and its chemical applications,  Gutman, Furtula and Petrovi\'{c} \cite{gutman2009} first proposed the concept of {\it terminal Wiener index}, which is defined as the sum of distances between all pairs of pendent vertices of trees. The terminal Wiener index is also arisen in the study of  phylogenetic tree reconstruction and the neighborhood of trees \cite{allen2001,humphries2008}.  For more information on the terminal Wiener indices, the readers may refer to
  the recent papers \cite{baskar2013,chen2013,deng2012,gutman2013,heydari2010,schmuck2012,szekely2011,xu2014} and the references cited therein.

  Let $T=(V(T), E(T))$ be a
tree of order $n$ with vertex set $V(T)$ and edge
set $E(T)$. The distance between vertices $v_i$ and $v_j$ is the
  number of edges in  the shortest path from  $v_i$ to $v_j$  and denoted by
$d_T(v_i,v_j)$ (or for short $d(v_i,v_j)$). Moreover,  {\it terminal Wiener index}  $TW(T)$ of a tree $T$ can be expressed as
\begin{eqnarray}
 TW(T)=\sum\limits_{ \{v_i,v_j\}\subseteq L(T)} d_T(v_i,v_j),
\end{eqnarray}
where  $L(T)$ is the set of pendent vertices in $V(T)$, i.e., the set of vertices with degree 1 in $V(T)$.
Gutman  et al. \cite{gutman2009} gave a formula for the terminal Wiener
index of trees
\begin{eqnarray}\label{eqn2}
 TW(T)=\sum\limits_{e=uv\in E(T)}p_u(e|T)p_v(e|T),
\end{eqnarray}
where $p_u(e|T)$ and $p_v(e|T)$ denote the number of pendent vertices of two
components of $T-e$ containing $u$ and $v$, respectively.
   The rest of the paper is organized as follows. In Section 2, we present a sharp upper and lower bounds for terminal Wiener index of a tree in terms of the number of vertices and diameter and characterize all extremal trees which attain these bounds.  In section 3, we investigate the properties of terminal Wiener index of a tree with fixed maximum degree.

\section{\large\bf{Trees with fixed diameter}}
   In this section, we only consider the terminal Wiener index of  $n-$vertex trees with a fixed diameter $d$.  Let $\mathcal{T}_{n,d}$ denote the set of all the trees of order $n$ with fixed diameter $d$ and let
$\mathcal{T}_{n,d,l}$ denote the set of all the trees of order $n$ with fixed diameter $d$ and the number $l$ of the  pendent vertices. Clearly, $\mathcal{T}_{n,d}$ consists of only star $K_{1, n-1} $ for $d=2$, and only  path for $d=n-1$.  Moreover, $2\le l\le n-d+1$ with the left equality holding if and only if $d=n-1$.
A  tree $T$ is called {\it caterpillar} if the graph from $T$ by deleting its all pendent vertices is  a path.
  A tree is called {\it starlike  tree of degree $k$ } if there is only one vertex with degree $k\ge 3$.  Gutman et al. \cite{gutman2009} presented the following result.
  \begin{theorem}\cite{gutman2009}\label{th1}
  Let $T$ be an $n-$vertex tree with  the number $l\ge 3$ of pendent vertices. Then
  \begin{equation}\label{th1-eq1}
  TW(T)\ge (n-1)(l-1)
  \end{equation}
  with equality if and only if $T$ is starlike of degree $l$.
  \end{theorem}
  In order to present the main result in this section, we need the following lemma.
  \begin{lemma}\label{lemma2}
  Let $T$ be an $n-$vertex tree with diameter $d$ and the number $l$ of the pendent vertices. Then
  \begin{equation}\label{le2-eq1}
  \lceil\frac{n-1}{\lfloor\frac{d}{2}\rfloor}\rceil\le l\le n-d+1, \ \ \ \rm{if} \ d\ \rm{is \ even};\end{equation}
   \begin{equation}\label{le2-eq2}
  \lceil\frac{n-2}{\lfloor\frac{d}{2}\rfloor}\rceil\le l\le n-d+1, \ \ \ \rm{if} \ d\ \rm{is \ odd}.\end{equation}
    \end{lemma}
\begin{proof} Let $P=v_0v_1\cdots v_d$ be a longest path
since the diameter of $T$ is $d$. For any vertex $u\in V(T)\backslash\{v_0, \cdots, v_d\}$, the distance between vertex $u$ and the path $P$ is at most $\lfloor\frac{d}{2}\rfloor$, i.e., ${\rm{dist}}(u, P)=\min\{d(u, v)\ | \ v\in V(P)\}\le \lfloor\frac{d}{2}\rfloor$. Otherwise the diameter of $T$ is larger than $d$.
Since every vertex in $u\in V(T)\backslash\{v_0, \cdots, v_d\}$ lies on a path from some pendent vertex except $\{v_0, v_d\}$ to the path $P$,   we have
$\lfloor\frac{d}{2}\rfloor(l-2)+d+1\ge n$. Therefore
if $d$ is even, then $\lfloor\frac{d}{2}\rfloor l\ge n-1$, i.e., $l\ge  \lceil\frac{n-1}{\lfloor\frac{d}{2}\rfloor}\rceil$;
if $d$ is odd, then $\lfloor\frac{d}{2}\rfloor l\ge n-2$, i.e., $l\ge  \lceil\frac{n-2}{\lfloor\frac{d}{2}\rfloor}\rceil$. So the assertion holds.
\end{proof}

Now we are ready to present a sharp lower bound for the terminal Wiener index of $n-$vertex trees with fixed diameter $d$.

\begin{theorem}\label{th2}
Let $T$ be an $n-$vertex tree with fixed diameter $d$, i.e., $T\in {\mathcal{T}}_{n, d}$. Then
\begin{equation}\label{th2-eq1}
TW(T)\ge (n-1)(l_0-1),
\end{equation}
where
$$l_0=\left\{\begin{array}{ll}
 \lceil\frac{n-1}{\lfloor\frac{d}{2}\rfloor}\rceil, &  {\rm{if}} \ d\ {\rm{is \ even}},\\
  \lceil\frac{n-2}{\lfloor\frac{d}{2}\rfloor}\rceil,  & {\rm{if}} \ d\ {\rm{is \ odd}}.
\end{array}\right.$$
Moreover, if $d\ge 3$, equality (\ref{th2-eq1}) holds if and only if $T$ is  starlike trees of degree $l_0$ and diameter $d$.
\end{theorem}
\begin{proof} If $d=n-1, $ the assertion holds. Assume that $d\le n-2$.
Let $T^*$ be an $n-$vertex tree with diameter $d$ such that $$
TW(T)\ge TW(T^*)\ \ \ \rm{ for }\ T\in  {\mathcal{T}}_{n,d}.$$
Denote by $l$ the number of pendent vertices of $T^*$.
By Lemma~\ref{lemma2}, $l\ge l_0\ge 3$.
On the other hand,
$T^*\in {\mathcal{T}}_{n, d, l}\subseteq {\mathcal{T}}_{n,l}$. Hence by Theorem~\ref{th1}, we have $TW(T^*)\ge (n-1)(l-1)$ with equality if and only if $
T^*$ is  starlike trees of degree $l$. Therefore
$$TW(T)\ge TW(T^*)\ge (n-1)(l-1)\ge (n-1)(l_0-1)$$
with equality if and only if $T$ is starlike trees of degree $l_0$ with diameter $d$.
\end{proof}

{\bf Remark} For given an $n$ and $d\le n-2$, there always exists at least one $n-$vertex starlike tree $T$ of degree $l_0$ with diameter $d$. For example, the $n-$vertex tree $T$ is obtained from $l_0-2$ paths of length $\lfloor\frac{d}{2}\rfloor$ and 2 paths of length
$\lceil\frac{d}{2}\rceil, n-\lfloor\frac{d}{2}\rfloor(l_0-2)-\lceil\frac{d}{2}\rceil-1,$ respectively,  by identifying one end of their paths. Moreover, the following result can be easily obtained from the proof of Theorem~5 in \cite{gutman2009}.

\begin{lemma}\cite{gutman2009}\label{lem4}
Let $g(x)=x(x-1)+(n-x-1)\lfloor\frac{x}{2}\rfloor \lceil\frac{x}{2}\rceil$ be positive integer function on $x$, where $n\ge 3$ is positive integer. Then
$g(x)$ is strictly increasing with respect to $2\le x\le \lfloor\frac{2n}{3}\rfloor+2; $ and strictly decreasing with respect to $\lfloor\frac{2n+1}{3}\rfloor+2\le x\le n-2.$
Moreover,
\begin{equation}\label{th5-eq1}
g(x)\le \left\{\begin{array}{ll}
\frac{1}{27}(n^3+9n^2+9n-27), & {\rm if}\  3\mid n; \\
\frac{1}{27}(n^3+9n^2+6n-16), & {\rm if} \ 3\mid (n-1);\\
\frac{1}{27}(n^3+9n^2+6n-2), & {\rm  if } \ 3\mid (n-2);\end{array}\right.
\end{equation}
with equality holding if and only if
$$x=\left\{\begin{array}{ll}
\lfloor\frac{2n}{3}\rfloor+2, & \rm{if}\  3\mid n;\\
\lfloor\frac{2n}{3}\rfloor+2 \ \rm{or}\ \lfloor\frac{2n+1}{3}\rfloor+2, & \rm{if }\ 3\mid (n-1);\\
\lfloor\frac{2n}{3}\rfloor+2, & \rm{if} \ 3\mid (n-2).\end{array}\right.$$
 \end{lemma}

\begin{theorem}\label{th3-d-max}
Let $T$ be an $n-$vertex tree with diameter $d$, i.e., $T\in \mathcal{T}_{n, d}$. If $d\ge\lfloor\frac{n-2}{3}\rfloor$, then
\begin{equation}\label{th3-eq1}
TW(T)\le (n-d+1)(n-d)+(d-2)\lfloor\frac{n-d+1}{2}\rfloor \lceil\frac{n-d+1}{2}\rceil.
\end{equation}
Moreover, if $n-d+1$ is even, then equality in (\ref{th3-eq1}) holds if and only if $T$ is
obtained from the path $P_{d-1}$ of order $d-1$ by attaching to each of its terminal vertices $\frac{n-d+1}{2}$ new pendent vertices and this tree is unique. If $n-d+1$ is odd, then equality in (\ref{th3-eq1}) holds if and only if $T$ is
obtained from the path $P_{d-1}$ of order $d-1$ by attaching to each of its terminal vertices $\lfloor\frac{n-d+1}{2}\rfloor$ new pendent vertices and by attaching one pendent vertex to some vertex of $P_{d-1}$ and there are $\lfloor\frac{d}{2}\rfloor$ distinct trees.
\end{theorem}
\begin{proof} If $n=3p$, let $l$ be the number of pendent vertices of $T$. Then  $l\le n-d+1
\le 3p-(p-1)+1=2p+2=\lfloor\frac{2n}{3}\rfloor+2$.  By Theorem~4 in \cite{gutman2009} and Lemma~\ref{lem4},
\begin{eqnarray*}
TW(T)&\le& l(l-1)+(n-1-l)\lfloor\frac{l}{2}\rfloor\lceil\frac{l}{2}\rceil\\
 &\le & (n-d+1)(n-d)+(d-2)\lfloor\frac{n-d+1}{2}\rfloor \lceil\frac{n-d+1}{2}\rceil.
 \end{eqnarray*}
 If equality in (\ref{th3-eq1}) holds, then by Lemma~\ref{lem4}, $l=n-d+1$. Moreover, by Theorem~4 in \cite{gutman2009}, all non-terminal edges $e=uv$,  we have $p_u(e|T)= \lfloor\frac{n-d+1}{2}\rfloor$  and  $p_v(e|T)= \lceil\frac{n-d+1}{2}\rceil$. Let $P_{d+1}=v_0v_1\cdots v_d$ be the longest path of $T$. Hence
  if $n-d+1$ is even, then for $e_1=v_1v_2$ and $e_{d-2}=v_{d-2}v_{d-1},$
   we have  $p_{v_1}(e_1|T)=\frac{n-d+1}{2}$ and  $p_{v_{d-1}}(e_{d-2}|T)=\frac{n-d+1}{2}$. So $T$ is
obtained from the path $P_{d-1}$ of order $d-1$ by attaching to each of its terminal vertices $\frac{n-d+1}{2}$ new pendent vertices and this tree is unique.
If $n-d+1$ is odd, then  $p_{v_1}(e_1|T)\ge \lfloor\frac{n-d+1}{2}\rfloor$ and  $p_{v_{d-1}}(e_{d-2}|T)\ge \lfloor\frac{n-d+1}{2}\rfloor$. Hence  $T$ is
obtained from the path $P_{d-1}$ of order $d-1$ by attaching to each of its terminal vertices $\lfloor\frac{n-d+1}{2}\rfloor$ new pendent vertices and by attaching one pendent vertex to some vertex of $P_{d-1}$ and there are $\lfloor\frac{d}{2}\rfloor$ distinct trees.  Conversely, it is easy to show that the equality holds.

 If $n=3p+1$ and $n=3p+2$, then by similar method, we can prove the assertion holds.
\end{proof}

{\bf Remark}  If $d=2$ or $d=3$, Theorem~\ref{th3-d-max} is still true. But if $4\le d<\lfloor\frac{n-2}{3}\rfloor$, Theorem~\ref{th3-d-max} is, in general, not true.  With aid of computing calculation, trees $ T_1, T_2, T_3, T_4$ (see Fig.1) have the largest terminal Wiener indices among all trees of order
$n=23$ with $ d=4$,   $n=30$ with $ d=5$, $n=40$ with $ d=6$, and $n=40$ with $ d=7$, respectively.

\setlength{\unitlength}{1mm}
\begin{picture}(140, 45)
\centering\thicklines \put(20,40){\circle*{1}}
\put(20,30){\circle*{1}} \put(5,30){\circle*{1}} \put(35,30){\circle*{1}}
\put(0,20){\circle*{1}}\put(10,20){\circle*{1}} \put(15,20){\circle*{1}}\put(25,20){\circle*{1}}
\put(30,20){\circle*{1}}\put(40,20){\circle*{1}}
\put(3,20){\circle*{0.6}}\put(5,20){\circle*{0.6}}\put(7,20){\circle*{0.6}}\put(18,20){\circle*{0.6}}\put(20,20){\circle*{0.6}}\put(22,20){\circle*{0.6}}
\put(33,20){\circle*{0.6}}\put(35,20){\circle*{0.6}}\put(37,20){\circle*{0.6}}
\put(20,40){\line(3,-2){15}} \put(20,40){\line(-3,-2){15}}\put(20,40){\line(0,-1){10}}
\put(5,30){\line(1,-2){5}} \put(5,30){\line(-1,-2){5}}\put(20,30){\line(1,-2){5}}\put(20,30){\line(-1,-2){5}}
\put(35,30){\line(1,-2){5}}\put(35,30){\line(-1,-2){5}}

\put(85,40){\circle*{1}} \put(70,35){\circle*{1}}  \put(100,35){\circle*{1}}
 \put(65,30){\circle*{1}}   \put(75,30){\circle*{1}} \put(95,30){\circle*{1}} \put(105,30){\circle*{1}}
 \put(61,20){\circle*{1}}\put(69,20){\circle*{1}}\put(71,20){\circle*{1}}\put(79,20){\circle*{1}}
\put(85,40){\line(-3,-1){15}}\put(85,40){\line(3,-1){15}}
\put(70,35){\line(-1,-1){5}}\put(70,35){\line(1,-1){5}}\put(100,35){\line(-1,-1){5}}\put(100,35){\line(1,-1){5}}
\put(65,30){\line(-2,-5){4}} \put(65,30){\line(2,-5){4}}\put(75,30){\line(2,-5){4}}\put(75,30){\line(-2,-5){4}}

\put(63.5,20){\circle*{0.6}}\put(65,20){\circle*{0.6}}\put(66.5,20){\circle*{0.6}}\put(73.5,20){\circle*{0.6}}\put(75,20){\circle*{0.6}}\put(76.5,20){\circle*{0.6}}
\put(98,30){\circle*{0.6}}\put(100,30){\circle*{0.6}}\put(102,30){\circle*{0.6}}
\put(1,19){$\underbrace{}_6$}\put(16,19){$\underbrace{}_6$}\put(31,19){$\underbrace{}_7$}
\put(61,19){$\underbrace{}_7$}\put(71,19){$\underbrace{}_8$}\put(96,29){$\underbrace{}_{10}$}
\put(8,5){\small{$TW(T_1)=582$ }}\put(70,5){\small{$TW(T_2)=1162$ }}

\end{picture}

\setlength{\unitlength}{1mm}
\begin{picture}(140, 45)
\centering\thicklines  \put(20,45){\circle*{1}}
\put(20,39){\circle*{1}} \put(5,39){\circle*{1}} \put(35,39){\circle*{1}}
\put(5,30){\circle*{1}}\put(20,30){\circle*{1}} \put(35,30){\circle*{1}}
\put(1,20){\circle*{1}}\put(9,20){\circle*{1}}\put(16,20){\circle*{1}}\put(24,20){\circle*{1}}\put(31,20){\circle*{1}}\put(39,20){\circle*{1}}
\put(3,20){\circle*{0.6}}\put(5,20){\circle*{0.6}}\put(7,20){\circle*{0.6}}
\put(18,20){\circle*{0.6}}\put(20,20){\circle*{0.6}}\put(22,20){\circle*{0.6}}
\put(33,20){\circle*{0.6}}\put(35,20){\circle*{0.6}}\put(37,20){\circle*{0.6}}
\put(20,45){\line(5,-2){15}} \put(20,45){\line(-5,-2){15}}\put(20,45){\line(0,-1){6}}
\put(5,39){\line(0,-1){10}}\put(20,39){\line(0,-1){10}}\put(35,39){\line(0,-1){10}}
\put(5,30){\line(-2,-5){4}}\put(5,30){\line(2,-5){4}}\put(20,30){\line(-2,-5){4}}\put(20,30){\line(2,-5){4}}\put(35,30){\line(-2,-5){4}}\put(35,30){\line(2,-5){4}}

\put(85,45){\circle*{1}}
\put(85,39){\circle*{1}} \put(70,39){\circle*{1}} \put(100,39){\circle*{1}}
\put(70,32){\circle*{1}}\put(85,32){\circle*{1}} \put(100,32){\circle*{1}}
\put(66,26){\circle*{1}}\put(74,26){\circle*{1}}\put(81,26){\circle*{1}}\put(89,26){\circle*{1}}\put(100,26){\circle*{1}}
\put(96,20){\circle*{1}}\put(104,20){\circle*{1}}
\put(68,26){\circle*{0.6}}\put(70,26){\circle*{0.6}}\put(72,26){\circle*{0.6}}
\put(83,26){\circle*{0.6}}\put(85,26){\circle*{0.6}}\put(87,26){\circle*{0.6}}
\put(98,20){\circle*{0.6}}\put(102,20){\circle*{0.6}}\put(100,20){\circle*{0.6}}
\put(85,45){\line(5,-2){15}} \put(85,45){\line(-5,-2){15}}\put(85,45){\line(0,-1){6}}
\put(70,39){\line(0,-1){7}}\put(85,39){\line(0,-1){7}}\put(100,39){\line(0,-1){7}}\put(100,32){\line(0,-1){6}}
\put(70,32){\line(-2,-3){4}}\put(70,32){\line(2,-3){4}}\put(85,32){\line(-2,-3){4}}\put(85,32){\line(2,-3){4}}
\put(100,26){\line(-2,-3){4}}\put(100,26){\line(2,-3){4}}

\put(1,19){$\underbrace{}_{11}$}\put(16,19){$\underbrace{}_{11}$}\put(31,19){$\underbrace{}_{11}$}
\put(66,25){$\underbrace{}_{10}$}\put(81,25){$\underbrace{}_{10}$}\put(96,19){$\underbrace{}_{12}$}
\put(8,6){\small{$TW(T_3)=2508$ }}\put(70,6){\small{$TW(T_4)=2592$ }}
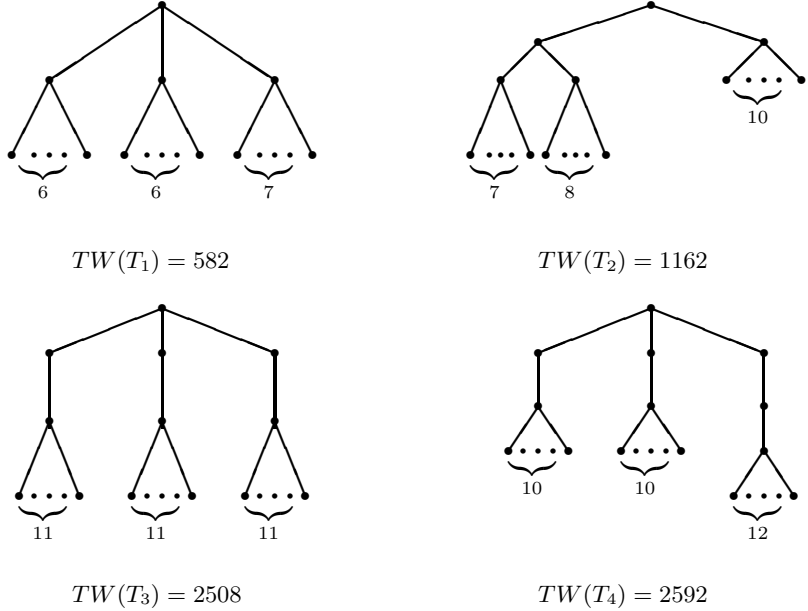
 \put(5,0){\small{Fig.1~~~$T_1, T_2, T_3, T_4$ have the maximum terminal Wiener index  }}
 \put(15, -5) {\small{  of trees  in $\mathcal{T}_{23,4}$,$\mathcal{T}_{30,5}$, $\mathcal{T}_{40,6}$ and $\mathcal{T}_{40,7}$}, respectively}
\end{picture}

\section{\large\bf{Terminal Wiener index with fixed maximum degree}}

  Let $\mathcal{T}_{n,\Delta}$ denote the set of all the trees of order $n$ and with maximum degree $\Delta$.
If $\Delta=2$, $\mathcal{T}_{n,2}$ consists of path $P_n$ of order $n$, and $TW(P_n)=n-1$. If $\Delta=n-1$, then $\mathcal{T}_{n,n-1}$ consists of star $K_{1,n-1}$ and $TW(K_{1,n-1})=(n-1)(n-2)$. By \cite{chen2013}, the extremal trees having the minimum terminal Wiener index in $\mathcal{T}_{n,\Delta}$  are starlike trees. It is natural to ask which are extremal trees having the maximum terminal Wiener index. Through this section, assume that $3\le \Delta\le n-2$.  A tree $T^*$ in  $\mathcal{T}_{n,\Delta}$ is called {\it optimal} tree if  $TW(T^*)\ge TW(T)$ for all $T\in \mathcal{T}_{n,\Delta}$. In this section, we discuss some properties of optimal trees.  Schmuck, Wagner and Wang \cite{schmuck2012} proved  the following result.
\begin{theorem}\cite{schmuck2012}\label{de-seq}
Let  $\mathcal{T}_\pi$ be the set of all trees with a given degree sequence $\pi=(d_1,d_2,\cdots,d_n)$ and $d_1\geq d_2\geq\cdots\geq d_k\ge 2>d_{k+1}=\cdots =d_n=1$.
If $d_2\ge 3$ and $TW(T^*)\ge TW(T)$ for any $T\in  \mathcal{T}_\pi$, then
$T^*$
is  an $n-$vertex caterpillar associated with $v_1, \cdots v_k$ vertices on the backbone of $T^*$ in this order with $d(v_i)=x_i+2$, $i=1, \cdots k$, and
\begin{eqnarray}\label{9}
 TW(T^*)=(n-1)(n-k-1)+F(x_1,\cdots,x_k),
 \end{eqnarray}
 where
$$ F(x_1,\cdots,x_k)=max\{F(y_1,\cdots,y_k)=\sum\limits_{i=1}^{k-1}(\sum\limits_{j=1}^iy_j)(\sum\limits_{j=i+1}^ky_j): y_1\geq y_k\}$$
and the maximum is taken over all permutations $(y_1,\cdots,y_k)$ of $(d_1-2,\cdots,d_k-2)$.
\end{theorem}

It follows from the method of \cite{schmuck2012} and \cite{zhang2010} that we are able to prove the following result.
\begin{lemma}
\label{lemma1.8}
  Let $w_1\geq w_2\geq\cdots\geq w_k\geq0$ be the integers with $k\geq5$  and let
$$F(x_1,\cdots,x_k)=max\{F(y_1,\cdots,y_k)=\sum\limits_{i=1}^{k-1}(\sum\limits_{j=1}^iy_j)(\sum\limits_{j=i+1}^ky_j): y_1\geq y_k\},$$
where  the maximum is taken over all permutations $(y_1,\cdots,y_k)$ of  $(w_1,\cdots,w_k)$. Then there exists a $2\leq t\leq k-2$ such that
the following holds:
\begin{eqnarray}\label{10}
x_1+x_2+\cdots+x_{t-1}\leq x_{t+2}+\cdots+x_k
\end{eqnarray}
and
\begin{eqnarray}\label{11}
x_1+x_2+\cdots+x_{t}> x_{t+3}+\cdots+x_k.
\end{eqnarray}
Further, if equation (\ref{10}) is strict, then
\begin{eqnarray}\label{12}
x_1\geq x_2\geq\cdots\geq x_{t-1}\geq x_t\geq x_{t+1}\leq x_{t+2}\leq\cdots\leq x_k;
\end{eqnarray}
if equation (\ref{10}) is equality, then
\begin{eqnarray}\label{13}
x_1\geq x_2\geq\cdots\geq x_{t-1}\geq x_t\geq x_{t+1}\leq x_{t+2}\leq\cdots\leq x_k
\end{eqnarray}
or
\begin{eqnarray}\label{14}
x_1\geq x_2\geq\cdots\geq x_{t-1}\geq x_t\leq x_{t+1}\leq x_{t+2}\leq\cdots\leq x_k.
\end{eqnarray}
\end{lemma}
\begin{proof}
By the definition of $F(x_1,\cdots,x_k)$, we get
\begin{eqnarray*}
0&\leq&F(x_1,\cdots,x_{i-1},x_i,x_{i+1},\cdots,x_k)-F(x_1,\cdots,x_{i-1},x_{i+1},x_i,\cdots,x_k)\\
  &=&(x_{i+1}-x_i)(\sum\limits_{j=1}^{i-1}x_j-\sum\limits_{j=i+2}^{k}x_j)\\
  &=&(x_{i+1}-x_i)f(i),
\end{eqnarray*}
where $f(i)=\sum\limits_{j=1}^{i-1}x_j-\sum\limits_{j=i+2}^{k}x_j$ for $1\leq i\leq k-1$.
Obviously, $f(1)<0$, $f(2)\leq0$, $f(k-1)>0$ and $f(i+1)\geq f(i)~~(1\leq i\leq k-1)$. Hence there exists a $2\leq t\leq k-2$ such that $f(t)\leq0$, $f(t+1)>0$, i.e (\ref{10}) and (\ref{11}) hold.

  Furthermore, $f(1)\leq f(2)\leq\cdots\leq f(t)\leq0<f(t+1)\leq f(t+2)\leq\cdots\leq f(k)$. If (\ref{10}) is strict, i.e $f(t)<0$, then
$$\sum\limits_{j=1}^{i-1}x_j<\sum\limits_{j=i+2}^{k}x_j ~~for~~1\leq i\leq t,$$
$$\sum\limits_{j=1}^{i-1}x_j>\sum\limits_{j=i+2}^{k}x_j ~~for~~t+1\leq i\leq k-1.$$
Hence, we obtain
$$x_{i+1}-x_i\leq0~~for ~~1\leq i\leq t$$
and
$$x_{i+1}-x_i\geq0~~for~~t+1\leq i\leq k-1,$$
which means
$$x_1\geq x_2\geq\cdots\geq x_t\geq x_{t+1}\leq x_{t+2}\leq\cdots\leq x_k$$
i.e (\ref{12}) holds.

If (\ref{10}) is equality, i.e $f(t)=0$, then exists a $1\leq s<t$ such that $f(1)\leq f(2)\leq\cdots\leq f(s)<f(s+1)=\cdots=f(t)=0$. Then we have
$$x_1\geq x_2\geq\cdots\geq x_t\geq x_{t+1}\leq x_{t+2}\leq\cdots\leq x_k$$
or
$$x_1\geq x_2\geq\cdots\geq x_t\leq x_{t+1}\leq x_{t+2}\leq\cdots\leq x_k$$
i.e (\ref{13}) or (\ref{14}) holds.
This completes the proof.
\end{proof}

\begin{theorem}
\label{max-de}
Let $\pi=(d_1,d_2,\cdots,d_n)$ with $d_1\geq\cdots\geq d_k\geq2\geq d_{k+1}=\cdots=d_1=1$ and $d_2\ge 3$, then
  if   $T^*$  is a maximum optimal tree in $\mathcal{T}_\pi$ with $F(x_1,\cdots,x_k)$ in equation (\ref{9}), then there exists a $2\leq t\leq k-2$  such that
  \begin{eqnarray}
\sum\limits_{i=1}^{t-1}{x_i}\leq \sum\limits_{i=t+2}^{k}{x_i}, \sum\limits_{i=1}^{t}{x_i}> \sum\limits_{i=t+3}^{k}{x_i}
\end{eqnarray}
and either $$x_1\geq x_2\geq\cdots\geq x_t\geq x_{t+1}\leq x_{t+2}\leq\cdots\leq x_k$$
or
$$x_1\geq x_2\geq\cdots\geq x_t\leq x_{t+1}\leq x_{t+2}\leq\cdots\leq x_k.$$
\end{theorem}
\begin{proof}
It follows from Theorem~\ref{de-seq} and Lemma~\ref{lemma1.8} that the assertion holds.
\end{proof}

\begin{corollary} \label{cor-max-de}
  Let $\mathcal{T}_{n,\Delta}$ denote the set of all the trees of order $n$ and with maximum degree $\Delta$. If $3\le \Delta\le n-3$ and $T^*$ is an optimal tree in $\mathcal{T}_{n,\Delta}$, then $T^*$ is an $n-$caterpillar tree and $v_1, \cdots, v_k$ vertices on the backbone of $T^*$ such that $d(v_i)=x_i+2$ and
there exists a $2\leq t\leq k-2$  such that
  \begin{eqnarray}
\sum\limits_{i=1}^{t-1}{x_i}\leq \sum\limits_{i=t+2}^{k}{x_i}, \sum\limits_{i=1}^{t}{x_i}> \sum\limits_{i=t+3}^{k}{x_i}
\end{eqnarray}
and either $$x_1\geq x_2\geq\cdots\geq x_t\geq x_{t+1}\leq x_{t+2}\leq\cdots\leq x_k$$
or
$$x_1\geq x_2\geq\cdots\geq x_t\leq x_{t+1}\leq x_{t+2}\leq\cdots\leq x_k$$
\end{corollary}
\begin{proof} If $\Delta=n-2$, then it is easy to see that the assertion holds.
 If $\Delta\le n-3$, denote by $\pi=(d_1, \cdots, d_n)$ the degree sequence of $T^*$ with $d_1\ge \cdots\ge d_n$. If $d_2=2$, $T^*$ is a starlike tree of degree $\Delta$ and $TW(T^*)=(n-1)(\Delta-1)$. The assertion holds. If $d_2\geq 3$, by Theorems~\ref{de-seq} and \ref{max-de}, the assertion also holds.
\end{proof}

\begin{lemma}\label{max-de-1}
Let $T^*$ be an optimal caterpillar  with $v_1, \cdots v_k$  vertices on the backbone of $T^*$ in the order and $d(v_i)$, $1\le i\le k$ satisfying $d(v_1)\ge\cdots \ge d(v_t)\ge 3$ and $3\le d(v_s)\le \cdots \le d(v_k),$ $t<s$. If $3\le \Delta\le n-3$, $d(v_{t-1})<\Delta$ and $d(v_s)<\Delta$, then $d(v_{t-2})=d(v_{s+1})=\Delta$, $d(v_t)=3$
 and $p_{v_{t-1}}(v_{t-1}v_t|T^*)-p_{v_t}(v_{t-1}v_t|T^*)+1=0$.
 \end{lemma}
\begin{proof}
Let $T_1$ be a caterpillar from $T^*$ by deleting one pendent edge at vertex $v_t$ and adding one pendent edge at vertex $v_{t-1}$.
  Let $T_2$ be a caterpillar from $T^*$ by deleting one pendent edge at vertex $v_t$ and adding one pendent edge at vertex $v_{s}$. Then
  \begin{equation}\label{max-de-1-eq1}
  TW(T^*)-TW(T_1)=p_{v_{t-1}}(v_{t-1}v_t|T^*)-p_{v_{t}}(v_{t-1}v_t|T^*)+1\ge 0
  \end{equation}
  and
   
  \begin{equation}\label{max-de-1-eq2}
  \begin{array}{l}
  TW(T^*)-TW(T_2)\ge (s-t)\{-2(d(v_t)-2)+2-(p_{v_{t-1}}(v_{t-1}v_{t}|T^*)\\
 \quad\quad\quad\quad\quad\quad\quad\quad\quad\quad -p_{v_{t}}(v_{t-1}v_{t}|T^*)+1)\}\ge 0
 \end{array}
  \end{equation}

  Hence by  (\ref{max-de-1-eq1}) and (\ref{max-de-1-eq2}), $d(v_t)=3$ and
  \begin{equation}\label{max-de-1-eq3}
  p_{v_{t-1}}(v_{t-1}v_t|T^*)-p_{v_{t}}(v_{t-1}v_t|T^*)+1=0.
   \end{equation}
   Suppose that $d(v_{t-2})<\Delta$. Then let  $T_3$ be a caterpillar from $T^*$ by deleting one pendent edge at vertex $v_{t-1}$ and adding one pendent edge at vertex $v_{t-2}$. Hence by (\ref{max-de-1-eq3}),
   \begin{eqnarray*}
   0&\le& TW(T^*)-TW(T_3)\\
   &=&-2(d(v_{t-1})-2)+p_{v_{t-1}}(v_{t-1}v_t|T^*)-p_{v_{t}}(v_{t-1}v_t|T^*)+1\\
   &=&-2(d(v_{t-1})-2)<0,
    \end{eqnarray*}
which is a contradiction. So $d(v_{t-2})=\Delta$.
   Suppose that $d(v_{s+1})<\Delta$. Then let  $T_4$ be a caterpillar from $T^*$ by deleting one pendent edge at vertex $v_{s}$ and adding one pendent edge at vertex $v_{s+1}$. Hence
   \begin{eqnarray*}
   0&\le&  TW(T^*)-TW(T_4)\\
   &=&-2((d(v_{t})-2)+\cdots+(d(v_{s})-2))- (p_{v_{t-1}}(v_{t-1}v_t|T^*)\\
   &&-p_{v_{t}}(v_{t-1}v_t|T^*)+1)+2\\
   &=&-2((d(v_{t})-2)+\cdots+(d(v_{s})-2))+2<0,
   \end{eqnarray*}
   which is a contradiction. So $d(v_{s+1})=\Delta$.
  \end{proof}

\begin{lemma}\label{max-de-2}
Let $T^*$ be an optimal caterpillar  with $v_1, \cdots v_k$  vertices on the backbone of $T^*$ in the order and $d(v_i)$, $1\le i\le k$ satisfying $d(v_1)\ge\cdots \ge d(v_t)\ge 3$ and $3\le d(v_s)\le \cdots \le d(v_k),$ $t<s$. If  $3\le \Delta\le n-3$, $d(v_{t-1})<\Delta$,  $d(v_s)=\Delta$ and $s>t+1$, then $d(v_{t-2})=\Delta$, $d(v_t)=3$
 and $p_{v_{t-1}}(v_{t-1}v_t|T^*)-p_{v_t}(v_{t-1}v_t|T^*)+1=0$.
 \end{lemma}
 \begin{proof}
 Let $T_5$ be a caterpillar from $T^*$ by deleting one pendent edge at vertex $v_t$ and adding one pendent edge at vertex $v_{t-1}$ and Let $T_6$ be a caterpillar from $T^*$ by deleting one pendent edge at vertex $v_t$ and adding one pendent edge at vertex $v_{t+1}$. Then
 \begin{equation}\label{max-de-2-eq1}
 TW(T^*)-TW(T_5)=p_{v_{t-1}}(v_{t-1}v_t|T^*)-p_{v_t}(v_{t-1}v_t|T^*)+1\ge 0
 \end{equation}
 and \begin{equation}\label{max-de-2-eq2}
 \begin{array}{l}
 TW(T^*)-TW(T_6)=-2(d(v_t)-3)-(p_{v_{t-1}}(v_{t-1}v_t|T^*)\\
 \quad\quad\quad\quad\quad\quad\quad\quad\quad\quad -p_{v_t}(v_{t-1}v_t|T^*)+1)\ge 0.
 \end{array}
 \end{equation}
 Hence by (\ref{max-de-2-eq1}) and (\ref{max-de-2-eq2}), we have
 $d(v_t)=3$ and $p_{v_{t-1}}(v_{t-1}v_t|T^*)-p_{v_t}(v_{t-1}v_t|\quad$  $ T^*)+1=0$.
  Suppose that $d(v_{t-2})<\Delta$. Then let  $T_7$ be a caterpillar from $T^*$ by deleting one pendent edge at vertex $v_{t-1}$ and adding one pendent edge at vertex $v_{t-2}$. Hence
  \begin{eqnarray*}
  0&\le& TW(T^*)-TW(T_7)\\
  &=&p_{v_{t-1}}(v_{t-1}v_t|T^*)-p_{v_t}(v_{t-1}v_t|T^*)+1-2(d(v_{t-1})-2)\\
&=&-2(d(v_{t-1})-2)<0,
   \end{eqnarray*}
which is a contradiction.
So $d(v_{t-2})=\Delta$. \end{proof}

 \begin{lemma}\label{max-de-3}
Let $T^*$ be an optimal caterpillar  with $v_1, \cdots v_k$  vertices on the backbone of $T^*$ in the order and $d(v_i)$, $1\le i\le k$ satisfying $d(v_1)\ge\cdots d(v_t)\ge 3$ and $3\le d(v_s)\le \cdots d(v_k),$ $t<s$. If  $3\le \Delta\le n-3$, $d(v_{t-1})<\Delta$,  $d(v_s)=\Delta$ and $s=t+1$, then $d(v_{t-3})=\Delta$.
 \end{lemma}
 \begin{proof}
 If $d(v_{t-2})=\Delta$, then $d(v_{t-3})=\Delta$. Hence assume that $d(v_{t-2})<\Delta$. let  $T_8$ be a caterpillar from $T^*$ by deleting one pendent edge at vertex $v_{t-1}$ and adding one pendent edge at vertex $v_{t-2}$ and let  $T_9$ be a caterpillar from $T^*$ by deleting one pendent edge at vertex $v_{t-1}$ and adding one pendent edge at vertex $v_{t}$. Then
  \begin{equation}\label{max-de-3-eq1}
  \begin{array}{l}
 TW(T^*)-TW(T_8)=p_{v_{t-1}}(v_{t-1}v_t|T^*)-p_{v_t}(v_{t-1}v_t|T^*)+1\\
 \quad\quad\quad\quad\quad\quad \quad\quad\quad\quad-2(d(v_{t-1})-2)\ge 0
 \end{array}
 \end{equation}
 and \begin{equation}\label{max-de-3-eq2}
 TW(T^*)-TW(T_9)=-p_{v_{t-1}}(v_{t-1}v_t|T^*)+p_{v_t}(v_{t-1}v_t|T^*)+1\ge 0.
 \end{equation}
 Hence by (\ref{max-de-3-eq1}) and (\ref{max-de-3-eq2}), we have
 $d(v_{t-1})=3$ and $p_{v_{t-1}}(v_{t-1}v_t|T^*)-p_{v_t}(v_{t-1}v_t|T^*)-1=0$.
 Suppose that $d(v_{t-3})<\Delta$. Then let  $T_{10}$ be a caterpillar from $T^*$ by deleting one pendent edge at vertex $v_{t-2}$ and adding one pendent edge at vertex $v_{t-3}$. Hence
 \begin{eqnarray*}
 0&\le&  TW(T^*)-TW(T_{10})\\
 &=&p_{v_{t-1}}(v_{t-1}v_t|T^*)-p_{v_t}(v_{t-1}v_t|T^*)+1-2(d(v_{t-2})+d(v_{t-1})-4)\\
 &=&-2(d(v_{t-2})-2)<0,
   \end{eqnarray*}
 which is a contradiction.
 Hence the assertion holds.
  \end{proof}
 \begin{theorem}\label{max-de-sum}
 Let $T^*$ be an optimal caterpillar  with $v_1, \cdots v_k$  vertices on the backbone of $T^*$ in the order and $d(v_i)$, $1\le i\le k$ satisfying $d(v_1)\ge\cdots \ge d(v_t)\ge 3$ and $3\le d(v_s)\le \cdots \le d(v_k),$ $t<s$. If  $3\le \Delta\le n-3$, then the following result holds.

 (1). If $d(v_{t-1})<\Delta, d(v_{s})<\Delta $, then $d(v_{t-2})=d(v_{s+1})=\Delta $.

 (2). If $d(v_{t-1})<\Delta, d(v_{s})=\Delta $ and $s>t+1$, then $d(v_{t-2})=\Delta $.

 (3). If $d(v_{t-1})<\Delta, d(v_{s})=\Delta $ and $s=t+1$, then $d(v_{t-3})=\Delta $.

 (4). If $d(v_{t-1})=\Delta, d(v_{s})<\Delta $, $d(v_t)<\Delta$ and $d(v_{s+1})<\Delta$, then $d(v_{s+2})=\Delta $.

 (5).   If $d(v_{t-1})=\Delta, d(v_{s})<\Delta $, $d(v_t)=\Delta$ and $d(v_{s+1})<\Delta$, then $d(v_{s+3})=\Delta $.
  \end{theorem}

\begin{lemma}
Let  $T$ be a caterpillar  with $v_1, \cdots v_k$  vertices on the backbone of $T$ in the order and $d(v_i)$, $1\le i\le k, k\ge 3$. If $d(v_1)=\cdots d(v_t)=3$, $d(v_s)=\cdots d(v_k)=3$,
then
\begin{equation}\label{max=3}
TW(T)=\frac{(l-1)(l^2+7l-12)}{6}+(t+1)(l-(t+1))(n+2-2l),
\end{equation}
where $l=n-k$ is the number of the pendent vertices of $T$.
\end{lemma}
\begin{proof} $s-t=n+3-2l$ and $l+k=n$. Moreover,
\begin{eqnarray*}
TW(T)&=& l(l-1)+2(l-2)+3(l-3)+\cdots +t(l-t)\\
&& + (t+1)[l-(t+1) ]+\cdots +(t+1)[l-(t+1) ]\\
&&+(t+2)[l-(t+2)]+\cdots +(l-2)[l-(l-2)]
\\
&=&\frac{(l-1)(l^2+7l-12)}{6}+(t+1)[l-(t+1)](n+2-2l).
\end{eqnarray*}
\end{proof}
\begin{lemma}\label{fun}
Let $g_1(x)=\frac{(x-1)(x^2+7x-12)}{6}+\frac{x^2}{4}(n+2-2x)$ and
$g_2(x)=\frac{(x-1)(x^2+7x-12)}{6}+\frac{x^2-1}{4}(n+2-2x)$. Then
$g_1(x)$ and $g_2(x)$ are strictly increasing with respect to $x$ in $x\in (1, \frac{n+4}{2})$.
\end{lemma}
\begin{proof} Note
$$g_1(x)=\frac{-4x^3+(3n+18)x^2-38x+24}{12},$$
$$g_1(x)^{\prime}=\frac{1}{12}(-12x^2+2(3n+18)x-38)>0$$
for $x\in(1, \frac{n+4}{2})$.
Hence $g_1(x)$ is strictly increasing with respect to $x$ in $x\in (1, \frac{n+4}{2}).$
Moreover,
$$g_2(x)=\frac{-4x^3+(3n+18)x^2-32x-3n+18}{12}.$$
Then
$$g_2(x)^{\prime}=\frac{1}{12}(-12x^2+2(3n+18)x-32)>0$$
for $x\in(1, \frac{n+4}{2})$.
Hence  $g_2(x)$ is strictly increasing with respect to $x$ in $x\in (1, \frac{n+4}{2})$
\end{proof}

\begin{theorem}
Let $T^*$ be an optimal tree in ${\mathcal{T}}_{n, 3}$ with $n\ge 6$.

(1).If $n=4p$, then $T^*$ is a caterpillar with $v_1, \cdots v_{2p-1}$  vertices on the backbone of $T$ in the order and $d(v_i)=3$ for $i=1, \cdots, 2p-1$. In other words,
$$TW(T)\le TW(T^*)=\frac{p(4p^2+18p-4)}{3} ~~\rm {for } ~T\in {\mathcal{T}}_{n, 3}$$
with equality if and only if $T$ is a caterpillar with $v_1, \cdots v_{2p-1}$  vertices on the backbone of $T$ in the order and $d(v_i)=3$ for $i=1, \cdots, 2p-1$.

(2). If $n=4p+1$, then $T^*$ is a caterpillar with $v_1, \cdots v_{2p}$  vertices on the backbone of $T$ in the order and $d(v_i)=3$ for $i=1, \cdots, p, p+2,\cdots, 2p$. In other words,
$$TW(T)\le TW(T^*)=\frac{p(4p^2+21p-1)}{3} ~~\rm {for } ~T\in {\mathcal{T}}_{n, 3}$$
with equality if and only if $T$ is a caterpillar with $v_1, \cdots v_{2p}$  vertices on the backbone of $T$ in the order and $d(v_i)=3$ for $i=1, \cdots,p, p+2, \cdots, 2p$.

(3). If $n=4p+2$, then $T^*$ is a caterpillar with $v_1, \cdots v_{2p}$  vertices on the backbone of $T$ in the order and $d(v_i)=3$ for $i=1, \cdots, 2p$. In other words,
$$TW(T)\le TW(T^*)=\frac{(2p+1)(2p^2+11p+3) }{3}~~\rm {for } ~T\in {\mathcal{T}}_{n, 3}$$
with equality if and only if $T$ is a caterpillar with $v_1, \cdots v_{2p}$  vertices on the backbone of $T$ in the order and $d(v_i)=3$ for $i=1, \cdots, 2p$.

(4). If $n=4p+3$, then $T^*$ is a caterpillar with $v_1, \cdots v_{2p+1}$  vertices on the backbone of $T$ in the order and $d(v_i)=3$ for $i=1, \cdots, p, p+2,\cdots, 2p+1$. In other words,
$$TW(T)\le TW(T^*)=\frac{(2p+1)(2p^2+11p+3) }{3}+(p+1)^2 ~~\rm {for } ~T\in {\mathcal{T}}_{n, 3}$$
with equality if and only if $T$ is a caterpillar with $v_1, \cdots v_{2p+1}$  vertices on the backbone of $T$ in the order and $d(v_i)=3$ for $i=1, \cdots,p, p+2, \cdots, 2p+1$.
\end{theorem}
\begin{proof}
Let $T^*$ be an $n-$vertex optimal tree in ${\mathcal{T}}_{n, 3}$. By  Corollary~\ref{cor-max-de}, $T^*$ is an $n-$caterpillar with $v_1, \cdots, v_k$ vertices on the backbone with $d(v_1)\ge \cdots d(v_t)\ge 3$ and
$3\le d(v_s)\le \cdots d(v_k)$, $1\le t<s\le k$. So, $d(v_i)=3, $ for $i=1, \cdots, t, $ and $i=s, \cdots, k$.
Denote by $l$ the number of pendent vertices of $T^*$. Then $l+k=n$ and $l\le k+2$, which implies that $l\le \frac{n+2}{2}$. By (\ref{max=3}), we have
\begin{eqnarray*}
TW(T^*)&=&\frac{(l-1)(l^2+7l-12)}{6}+(t+2)(l-(t+2))(n+2-2l)\\
&\le& \frac{(l-1)(l^2+7l-12)}{6}+\lfloor\frac{l}{2}\rfloor \lceil\frac{l}{2}\rceil(n+2-2l)
\end{eqnarray*}
If $n=4p$, then by Lemma~\ref{fun},   $$TW(T^*)\le g_2(2p+1)=\frac{p(-8p^2+(3n+6)p+3n-4)}{3}=\frac{p(4p^2+18p-4)}{3}$$
Hence the assertion holds.
If $n=4p+1, 4p+2, 4p+3$, then by the same argument, the assertion holds.
 \end{proof}


\frenchspacing
\begin{thebibliography}{}


\bibitem{allen2001}B.~Allen, M~Steel, Subtree transfer operations and their induced metrics on evolutionary trees, {\it Ann.Comb.} 5 (2001) 1-15.

\bibitem{baskar2013}B.~Baskar Babujee, J.~Senbagamalar, Terminal Wiener index for graph structures, {\it International Journal of Mathematical Sciences} 7 (2013) 39-42.



\bibitem{chen2013}Y.-H.~Chen, X.-D.~Zhang,   On Wiener and terminal Wiener
indices of trees, {\it MATCH Commun Math. Comput. Chem.}  70 (2013) 591-602.


\bibitem{dobrynin2001}A.~A.~Dobrynin, R.~Entringer and I.~Gutman,
Wiener index of trees: theory and applications, {\it Acta Appl.
Math.} 66 (2001) 211-249.

\bibitem{deng2012}X.~Deng, J.~Zhang,
Equiseparability on terminal Wiener index, {\it Applied Mathematics Letters} 25 (2012) 580-585.



\bibitem{Gutman1997-k}I.~Gutman, S.~Klavzar, B.~Mohar(eds.), Fifty years of the Wiener
index, {\it MATCH Commun Math. Comput. Chem.}  35 (1997) 1-159.

\bibitem{gutman-k-2}I.~Gutman, S.~Klavzar, B.~Mohar (eds.), Fiftieth anniversary of the
Wiener index. {\it Discrete Appl. Math.}  80 (1997) 1-113.

\bibitem{gutman2009}I.~Gutman, B.~Furtula, M.~Petrovi\'{c}, Terminal Wiener index,
{\it J. Math. Chem.}  46 (2009) 522-531.

\bibitem{gutman2013}I.~Gutman, B.~Furtula, J.~To$\check{s}$ovi\'{c}, M.~Essalih, M.~El Marraki,  On the terminal Wiener indices of kenograms and plerograms,
{\it Iranian Journal of Mathematical Chemistry}.  4 (2013) 77-89.


\bibitem{heydari2010}A.~Heydari, I.~Gutman, On the terminal Wiener index of thorn graphs,
{\it Kragujevac J. Sci.}  32 (2010) 57-64.

\bibitem{horvat2008}B.~Horvat, T.~Pisanski, M.~Randi\'{c}, Terminal polynomials and star-like graphs,
{\it MATCH Commun. Math. Comput. Chem.}  60 (2008) 493-512.

\bibitem{humphries2008}P.~Humphries,  Combinatorial aspects of leaf-labelled trees, University of Canterbury,Ph.D. Thesis, 2008. http://hdl.handle.net/10092/1801.




\bibitem{randic2007}M.~Randi\'{c}, J.~Zupan, D.~Viki\'{c}-Topi\'{c}, On representation of proteins by starlike graphs,
{\it J. Mol. Graph. Modell}  26 (2007) 290-305.

\bibitem{randic2004}M.~Randi\'{c}, J.~Zupan, Highly compact 2D graphical representation of DNA sequence,
{\it SAR QSAR Environ. Res}  15 (2004) 191-205.


\bibitem{schmuck2012}N.S.~Schmuck, S.G.~Wagner, H.~Wang, Greedy trees, caterpillars, and
Wiener-type graph invariants, {\it MATCH Commun. Math. Comput. Chem.}  68 (2012) 273-292.

\bibitem{smolenskii2009}E.~A.~Smolenskii, E.~V.~Shuvalova, L.~K.~Maslova, I.~V.~Chuvaeva,
M.~S.~Molchanova, Reduced matrix of topological distances with a minimum number of
independent parameters: distance vectors and molecular codes, {\it J. Math. Chem.} 45 (2009) 1004-1020.

\bibitem{szekely2011}L.A.~Sz\'{e}kely, H.~Wang, T.~Wu, The sum of the distances between the leaves of a tree and the
`semi-regular' property, {\it Discrete Math.}  311 (2011) 1197-1203.

\bibitem{wang2008}S.~Wang, X.~Guo, Trees with extremal Wiener
indices {\it MATCH Commun Math. Comput. Chem.}  60 (2008) 609-622.

\bibitem{wiener1947}H.~Wiener,  Structural determination of paraffin
boiling points, {\it J. Amer. Chem. Soc.}  69 (1947) 17-20.

\bibitem{xu2014}K.~Xu, M.~Liu, K.~C.~Das, I.~Gutman, B.~Furtula,  A Survey on graphs extremal with respect to
distance-based topological indices, {\it MATCH Commun Math. Comput. Chem.}  71 (2014) 461-508.

\bibitem{zhang2010}X.-D.~Zhang, Y.~Liu, M.~X.~Han, Maximum Wiener index of trees with given degree sequence, {\it MATCH Commun Math. Comput. Chem.}  64 (2010) 661-682.
\end{thebibliography}
\end{document}